\documentclass[12pt]{amsart}
\usepackage{amsfonts}
\usepackage{latexsym}
\usepackage{amsmath}
\usepackage{amsthm}
\usepackage{amscd,enumerate}
\usepackage{graphicx}
\usepackage{amsfonts}
\usepackage{amssymb}
\usepackage[mathscr]{eucal}

\newtheorem{theorem}{Theorem}[section]
\newtheorem*{prop1}{Proposition 1}
\newtheorem*{prop2}{Proposition 2}
\newtheorem*{thm1}{Theorem 1}
\newtheorem*{thm2}{Theorem 2}
\newtheorem*{thm3}{Theorem 3}
\newtheorem{lemma}[theorem]{Lemma}

\newcommand{\HH}{\mathcal{H}}
\newcommand{\LH}[1]{\Lambda_R(n,\HH_{#1},k_{#1}+l_{#1})}
\newcommand{\lH}[2]{\lambda_R(d_{#1},k_{#2}+l_{#2})}
\newcommand{\lo}[1]{\Omega_{#1}}
\newcommand{\gs}{\mathfrak{S}}

\newcommand{\bv}{{\bf v}}

\title{The Goldston-Pintz-Y{\i}ld{\i}r{\i}m Sieve and Maximal Gaps}
\author{Hakan Ali-John Seyalioglu}
\address{Department of Mathematics, University of California - Los Angeles, Los Angeles, CA 90024, USA}
\email{hseyalioglu@ucla.edu}
\thanks{This research was conducted under the support of a Fulbright Student Grant administered
 by the Hungarian Fulbright Commission while visiting the Alfr\'{e}d R\'{e}nyi Mathematical Institute
  of the Hungarian Academy of Sciences. }
\begin{document}
\maketitle
\section{Introduction}

One field of particular interest in Number Theory concerns the gaps between consecutive primes. Within the last
few years, very important results have been achieved on how small these gaps can be. The strongest of these
results were obtained by Dan Goldston, J{\'a}nos Pintz and Cem Yal\c{c}{\i}n Y{\i}ld{\i}r{\i}m. The present work
begins by generalizing their results so that they can be applied to related problems in a more direct manner.
Additionally, we improve the bound for $F_2$ (concerning the maximal gap in a block of three primes) obtained by
the results of \cite{goldston-2005} with our generalization.

\subsection{Previous Work}

The  first result of Goldston, Pintz and Y{\i}ld{\i}r{\i}m \cite{goldston-2005} states  \footnote{In earlier
works \cite{huxley-1969}, the constant $\Delta_r$ is referred to as $E_r$. However, this has been dropped in
more recent works \cite{goldston-2005, goldston-2006} to avoid a notational conflict.}
\begin{equation*}
\Delta_r = \liminf_{n \to \infty} {p_{n+r} - p_n \over \log{p_n}} \leq(\sqrt{r}-1)^2
\end{equation*}
\noindent and in particular $\Delta_1 = 0$. Using related methods and
incorporating ideas from Maier's matrix method \cite{maier-1988}, the
authors were able to later improve this result by a factor of $e^{-
\gamma}$ \cite{goldston-2006}. In proving this result, the authors
developed a new sieve method, based closely on that of Selberg, to
estimate the number of primes within an interval. Let $\HH = \{h_1, h_2,
\ldots h_k \} $ and $P(n, \HH) = (n+h_1)(n+h_2) \ldots (n+h_k)$. Using the
notation of \cite{motohashi-2005}, for $l \geq 0$:
\begin{align*}
\Lambda_R(n,\HH,k+l)= {1 \over
(k+l)!}\displaystyle\sum_{\substack{d|P(n,\HH) \\ d \leq R} }
\mu(d)\log\left({R \over d} \right)^{k+l}.
\end{align*}

The main results of \cite{goldston-2005} follow from two related
estimates, enumerated below. Throughout this paper $C$ and $c$ are
absolute constants which may differ at every occurrence. If an implied
constant depends on a value this dependence will be denoted with a
subscript of the value the constant depends on (for example $\ll_M$,
$o_M(1)$, $C_M$ denote dependence on $M$). Define $\HH = \HH_1 \cup \HH_2
\subset [1, 2 , \ldots h]$, $|\HH_i| = k_i$, $|\HH_1 \cap \HH_2| = r$, $M
= k_1 + k_2 + l_1 + l_2$.

\begin{prop1} \label{prop1} If $R \ll N^{1 \over 2}(\log N)^{-4M}$ and $h \leq R^{C}$ for any
$C > 0$, then, as $R, N \to \infty$, \\
\noindent $\displaystyle \sum_{N<n \leq 2N} \LH 1 \LH 2 = $ \vspace{-.1in}\\

$\hfill \displaystyle {l_1 + l_2
\choose l_1}{(\log R)^{r + l_1 + l_2} \over {r + l_1 + l_2!}} ( \gs(\HH) +
o_M(1))N.$
\end{prop1}

With this proposition, the authors are able to understand how the weights
behave in an unmodified fashion over an interval. Their second task is to
see how these weights are modified by incorporating the $\theta$ function
($\theta(n) = \log n$ if $n$ is a prime and $0$ otherwise).

\begin{prop2} \label{prop2} Take $h_0 \not \in \HH$. If $R \ll N^{1 \over 4}(\log
N)^{-B_M}$ for a sufficiently large constant $B_M$, and $h \leq R$
then, as $R, N \to \infty$, \\

\noindent $\displaystyle \sum_{N<n \leq 2N} \theta (n + h_0) \LH 1 \LH 2 = $ \vspace{-.1in}\\

$\hfill \displaystyle {l_1 + l_2 \choose l_1}{(\log R)^{r + l_1 + l_2}
\over {r + l_1 + l_2!}} ( \gs(\HH \cup h_0) + o_M(1))N.$
\end{prop2}

With the observation that if $n+ h_0$ is a prime and $h_0 \in \HH$,
$\Lambda_R(n,\HH,k+l)= \Lambda_R(n,\HH \setminus {h_0} ,(k-1)+(l+1))$ it
is then possible to work around the restriction that $h_0 \not \in \HH$,
providing a result for general $h_0$. \vspace{-.1in}

\subsection{Results}

Summing over all values for $h_0$ in an interval, the authors are able to
get a sum of the logarithms of all primes that lay in the designated
interval when multiplied by the $\Lambda$ weight functions. However, a
natural problem is how to sum over more complex distributions of primes.
Instead of a sum over all primes in an interval, consider the problem of
wanting a sum over all primes $p$ such that $p+j_1$ and $p+j_2$ are also
primes. This problem is too complex for a strict asymptotic in this
fashion, however, by modifying the sieve to deal with four $\Lambda$
functions instead of two, we give a method by which an upper bound can be
reached for the logarithms of such pairs of primes multiplied by the
weight function. Say we need an upper bound for: \vspace{-.15in}

\begin{equation} \label{eq1} \sum_{n = N+1} ^ {2N} \theta(n+h_0) \theta(n+h_1)
\left( \sum_{\HH} \LH{} \right)^2.  \end{equation}

The $\Lambda$ function behaves very predictably when $n+h_0$ and $n + h_1$
are prime. Consulting the definition, one quickly derives that (assuming $h, R < N$):
\begin{equation*}
\theta(n+h_0) \theta(n+h_1) \leq \log ^2 3N \left( \Lambda(n, \{h_0,
h_1\}, 2)  {2 \over \log^2 R} \right) ^2.
\end{equation*}
Additionally, if $n + h_0$ and $n+ h_1$ are both prime, $h_0$ and
$h_1$ do not effect the second $\Lambda$ function.
Therefore, letting $\Omega := 4 \log^2 (3N) / \log^4 R$,
$$\displaystyle (\ref{eq1}) \leq  \Omega  \sum_{n = N+1} ^ {2N}
\Lambda(n, \{h_0, h_1\}, 2)  ^2  \left( \sum_{\HH} \Lambda(n, \HH \setminus \{h_0, h_1 \}, k +l) \right)^2 . $$

From this we can see that one way to attack the previously mentioned and similar problems is to understand how
four $\Lambda$ functions act when the first two and second two take
disjoint sets as their second arguments. {\it Theorem 1} addresses this
problem. Define $|\HH_i| = k_i$, $\HH_i
\subset [1, h]$, $|\HH_1 \cap \HH_2| = r_1$, $|\HH_3 \cap \HH_4| = r_2$ and $M = \sum_{i=1}^4 k_i + l_i$.
\begin{thm1} \label{resultthm}
 Let $(\HH_1 \cup \HH_2)
\cap (\HH_3 \cup \HH_4) = \emptyset$, $R \ll N^{1/4} \log(N)^{-C_M}$ for a sufficiently large $C_M$ and for any $C >0$, $h \ll R^C$.
Then, letting $u = l_1 + l_2 + r_1$, $v = l_3 + l_4 + r_2$ as $R, N \to
\infty$, \\

$\label{original} \displaystyle \hspace{-.15in} \sum_{N<n \leq 2N} \hspace{-.15in}\LH 1 \LH 2 \LH 3 \LH4
= $
$$\displaystyle \hfill {l_1 + l_2 \choose l_1}{l_3 + l_4 \choose l_3}N {(\log R)^{u + v}
\over {u!v}!} \left( \gs(\HH_1 \cup \HH_2 \cup \HH_3 \cup \HH_4) + o_M(1) \right). $$
\end{thm1}

When applying this sieve result however, a second problem presents itself if one does not want the $\HH_i$ sets all taken uniformly from the same
interval. In applying the
Propositions of \cite{goldston-2005} a result of Patrick
Gallagher \cite{gallagher-1976} is used on the average of singular series when
the sets under consideration are taken uniformly from an interval.
But, if both pairs of sets are taken from different intervals, their
union (which is considered in the singular series) may not be uniformly
varying over all $k$ element subsets of a given interval, but instead over a
more complex distribution (e.g., over all sets $\HH \subset [1,h]$ with two elements from $[1, h']$ and one element from $[h',h'']$). It is with this in mind that we present {\it Theorem 2}, a more
general version of Gallagher's result. We show that instead of varying the
set uniformly over one interval, if we instead take several subsets which
vary uniformly over subintervals, the singular series of their union will
still average to $1$ assuming the subintervals obey a certain growth condition. Kevin Ford's \cite{ford-2007} recent simplification
of Gallagher's proof is the foundation for the extension presented. Let
$\sum_{i=1}^l k_i = r$ and $\Omega_{\mathcal{H}}(p)$ be the residue classes
occupied by the elements of $\mathcal{H}$ modulo $p$.

\begin{thm2}

Take an interval $[0, h]$ and take $l$ subintervals $[B_i(h), C_i(h)] \subset
[0, h]$ where $C_i(h) - B_i(h) = d_i (h)$. Assume for some $1 > \delta > 0$,
and for all $i \leq l$, $h^\delta = o(d_i (h))$. Then, as $h \to \infty$,
\begin{equation*} \label{originaleq}
 \sum_{\substack{A_1, A_2, \ldots A_l: \\ A_i \subset [B_i(h), C_i(h)], |A_i| = k_i}} \gs \left( \bigcup_{i =1 }^l A_i \right) = \prod_{i = 1}^l {d_i^{\;
k_i}  (h) \over k_i !} \left( 1 + o_{r, \delta}(1)) \right).
\end{equation*}
\end{thm2}
Finally, we apply these two results into a concrete application. Restating
the usual definition (notice $F_n$ is trivially bounded above by $\Delta_n$):
$$
F_r = \liminf_{n \to \infty} { \max_{1 \leq i \leq r} {p_{n +i} - p_{n
+ i - 1} \over \log{p_n}}}.
$$

\begin{thm3} There exists a $c > 0$ such that $F_2 \leq c < (\sqrt{2} - 1)^2$.  \end{thm3}
Where the $c$ is explicitly computable and we give such a $c$. While this
falls short of improving the best known result for $F_2$, which is the current best bound for $\Delta_2$ given
in \cite{goldston-2006} as $ e^{- \gamma} (\sqrt{2} - 1)^2$, since the previous bound is the
result for $\Delta_2$ of \cite{goldston-2005} with Maier's matrix method
applied to it, it seems likely that a similar application of the matrix
method would provide a corresponding improvement. 
The best $c$ our method gives is approximately $.1707$, a modest
improvement over \cite{goldston-2005}'s $.1716$. Our proof relies on breaking the interval considered in \cite{goldston-2005}
into three and modifying the weighted difference to ensure that positivity implies either a prime occurring in the middle interval or three primes occurring in an end interval. {\it Theorems 1} and {\it 2}
are provided with precise error terms in their corresponding sections. \\

The only widely available work which has managed to distance the best known
bounds for $\Delta_n$ and $F_n$ for any $n$ is that of Huxley \cite{huxley-1969} (for the case $n=2$ it was shown approximately that $\Delta_2 \leq 1.4105$ and $F_2 \leq 1.3624$).
The lack of other results should not imply a disinterest in the $F_n$ constants; the question of $F_2 < 1$ is attributed to Erd\H{o}s \cite{maier-1988}. As distancing the $F_n$ and $\Delta_n$ constants
have proven very difficult, Erd\H{o}s' problem was not resolved until it was shown that $\Delta_2 < 1$ by Maier applying his matrix method to Huxley's results \cite{maier-1988}.
Currently the best bound we have for $F_2$ is the trivial one afforded by $\Delta_2$ in \cite{goldston-2006}. This paper is a first step in distancing the two constants, with the only
 conjectured additional result needed to give an improvement being a successful application of Maier's method \cite{maier-1985}. It should be stressed that
 applying Maier's method is not at all trivial (the entire subjects of \cite{goldston-2006} and \cite{maier-1988} are applying the method to \cite{goldston-2005} and \cite{huxley-1969} respectively).
\subsection{Acknowledgements} I would like to give profound thanks to Professors Antal Balog and Andr\'{a}s Bir\'{o} for
supervising this project and for all their time and guidance during my time in Budapest. I was originally given
this problem by Professor J\'{a}nos Pintz and his comments during the project were invaluable. I also owe a great deal to the R\'{e}nyi Institute's hospitality and the Hungarian Fulbright Commission's generosity. Many helpful comments and corrections were also provided by the anonymous reviewer.

\section{Proof of {Theorem 1}}

\begin{thm1} \label{resultthm}
 Let $(\HH_1 \cup \HH_2)
\cap (\HH_3 \cup \HH_4) = \emptyset$, and for any $C >0$, $h \ll R^C$.
Then,  for any $\gamma > 0$, letting $u = l_1 + l_2 + r_1$, $v = l_3 + l_4
+ r_2$ and as $R, N \to \infty$, \vspace{.1in}

$\hspace{-.2in}\label{original}\displaystyle\sum_{N<n \leq 2N} \hspace{-.1in} \LH 1 \LH 2 \LH 3 \LH4
= $

$$\displaystyle {l_1 + l_2 \choose l_1}{l_3 + l_4 \choose l_3}N
{(\log R)^{u + v} \over {u!v}!}\gs(\bigcup_{0<i\leq4} \HH_i) + $$

$\displaystyle \hfill O_{M,
\gamma}(N(\log N)^{u + v - 1 + \gamma} + R^4 (\log R)^{C_M}). $

\end{thm1}

\subsection{Outline}

The proof of {\it Theorem 1} follows the same main outline as the proof of
{\it Proposition 1} from \cite{goldston-2005}, with a few alterations to
allow four weights instead of two. It will be necessary to use {\it Lemma
3} from \cite{goldston-2005} in the analysis, which is stated in {\it
Section \ref{GPYSection}}. We outline the proof below.

\begin{enumerate}
\item Translate the product of $\Lambda$ functions into to a a complex integral and translate part of the integrand into an Euler product - {\it
Section \ref{RewriteProd}}
\item Estimate the error term from the translation - {\it
Section \ref{TheFirstError}}
\item We are now left with a complex integral over four variables to
estimate. We introduce a series of zeta functions which estimate the
function and prove a lemma on how well these product of zeta functions
estimate our integrand - {\it Lemma \ref{GLemma}} - It is at this point
that our assumption about the disjointness of the unions of the two pairs
of sets is vital - {\it Section \ref{ZetaSection}}
\item By our choice of the zeta weights we are almost able to separate the integral over four variables
into two double integrals, however, there is some interplay in the $G$
function (which represents the error with which the product of the zeta
functions estimate our product). We show that the $G$ function is small
enough when two of the variables are fixed and non-negative to use {\it
Lemma \ref{GPYLemma}} twice - {\it Section \ref{SplittingSection}}
\end{enumerate}

\subsection{Rewriting the Product} \label{RewriteProd}
First let,
$$\lambda_R(d;a) =\left\{
  \begin{array}{ll}
    0 & \hbox{\text{ if } $d>R$;} \\
    \displaystyle {1 \over a!} \mu(d) \left( \log {R \over d}\right)^a & \hbox{\text{ if } $d \leq R$}.
  \end{array}
\right.$$ If we let $\Omega_i(p)$ be defined as the set of different
residue classes among $-h \mod p$ where $h \in \HH_i$ and extend it
multiplicatively (as in \cite{motohashi-2005} and \cite{goldston-2005}), \\

$\displaystyle \hspace{-.2in} \sum_{N<n \leq 2N} \hspace{-.1in} \LH 1 \LH 2 \LH 3 \LH 4$ \bigskip

$\displaystyle \hspace{-.1in}= \hspace{-.25in}\sum_{d_1, d_2, d_3, d_4} \hspace{-.1in} \lH 1 1 \lH 2 2 \lH 3 3 \lH 4 4
\hspace{-.3in} \sum_{\substack{N < n \leq 2N \\ n \in \Omega_1(d_1), n \in
\Omega_2(d_2)
\\ n \in \Omega_3(d_3), n \in \Omega_4(d_4)}}  \hspace{-.3in} 1.
$

From this one derives that: \\

$\displaystyle \hspace{-.15in}\sum_{N<n \leq 2N} \hspace{-.15 in} \LH 1 \LH 2 \LH 3 \LH 4 = $ 

$N T + T'$ where $\displaystyle T = \sum_{d_1, d_2, d_3, d_4}  {|\Phi(d_1, d_2, d_3,
  d_4)| \over [d_1, d_2, d_3, d_4]} \times$

$ \hfill \lH 1 1 \lH 2 2 \lH 3 3 \lH 4 4 $, and \\

$\displaystyle  T' = O( \sum_{d_1, d_2, d_3, d_4} |\Phi(d_1, d_2, d_3, d_4)| \times $\\

$\displaystyle \hfill |\lH 1 1 \lH 2 2 \lH 3 3 \lH 4 4 | ). $ \\

With $\Phi(\cdot, \cdot, \cdot, \cdot)$ defined over prime quadruples,
where $\beta_i \in \{0,1 \}$, \\ $\Phi (p^{\beta_1}, p^{\beta_2},
p^{\beta_3}, p^{\beta_4})$ = $|  \bigcap_{i : \beta_i = 1} \lo i (p) |$
and extended multiplicatively, where $\beta_i = 1$ if $p | n_i$ and $0$
otherwise, $\Phi(n_1, n_2, n_3, n_4)  = \prod_p \Phi(p^{\beta_1},
p^{\beta_2}, p^{\beta_3}, p^{\beta_4}) $.

\subsection{The First Error Term} \label{TheFirstError}In this section we will show \\ $T' = O_M(R^4 (\log
R)^{C_M})$, giving the second error term in {\it Theorem 1.} Notice that
the $\lH { } { }$ factors can be bound with a constant power of $\log R$
depending only on $k$ and $l$. It remains to bound:

$\displaystyle \sum_{d_1, d_2, d_3, d_4 < R} \hspace{-.2in} \Phi(d_1, d_2, d_3, d_4)$,
which is bounded above by $ \displaystyle \prod_{0< i \leq 4} \left[
\sum_{d < R} |\lo i (d)| \right]^4.$

Which is bounded by $R^4 (\log R)^{C_M}$ because $|\lo i (d)|$ is bounded by
the $k_i^{\text{th}}$ generalized divisor function. It is worth noting the
$R^{4}$ in this error term. In all applications which do not assume
Elliot-Halberstam type results on the distribution of primes, Goldston,
Pintz and Y{\i}ld{\i}r{\i}m take $R = N^{{1 \over 4} - \epsilon}$ due to
bounds given by the Bombieri-Vinogradov theorem. In our application, we
have this restriction in an unrelated point in the analysis. 

\subsection{Introducing the Zeta Weights} \label{ZetaSection}

The next step is to write the formula as an Euler product. Using the
complex analytic equality (the integral is taken over $s$ with $\Re(s) = 1$):

\begin{equation*}
\lambda_R(d;a) = {\mu(d) \over 2 \pi i} \int_{(1)} \left( R \over d
\right) ^s {ds \over s^{a+1}},
\end{equation*}

$$ T = { 1 \over (2 \pi i)^4} \int_{(1)} \int_{(1)} \int_{(1)} \int_{(1)} {F(s_1, s_2, s_3, s_4) R^{s_1 +
    s_2 + s_3 + s_4} \over s_1^{k_1 + l_1 +1} s_2^{k_2+l_2+1} s_3^{k_3+l_3+1}s_4^{k_4+l_4+1}}ds_1 ds_2 ds_3 ds_4.$$

\noindent Where,
$$ F(s_1, s_2, s_3, s_4) = \displaystyle\sum_{d_1, d_2, d_3,
  d_4}\mu(d_1)\mu(d_2)\mu(d_3)\mu(d_4){|\Phi(d_1, d_2,d_3, d_4)| \over
  [d_1, d_2, d_3, d_4] d^{s_1}_1 d^{s_2}_2 d^{s_3}_3 d^{s_4}_4} = $$

\begin{equation*}
\displaystyle\prod_p  [ 1 - {|\lo 1 (p)|\over
      {p^{s_1+1}}} - {|\lo 2 (p)|\over {p^{s_2+1}}}+ {|\lo 1 (p) \cap \lo 2 (p)|\over {p^{s_1 + s_2+1}}}
- {|\lo 3 (p)| \over {p^{s_3 +1}}} - {|\lo 4 (p)| \over {p^{s_4 +1}}} +
    {|\lo 3 (p) \cap \lo 4 (p)|\over {p^{s_3 + s_4 +1}}} \end{equation*}
\begin{equation*}
+ {|\lo 1(p) \cap \lo 3(p)|\over {p^{s_1 + s_3 +1}}} + {| \lo 2 (p) \cap
\lo 3 (p)| \over {p^{s_2 + s _3 + 1}}} +
  {|\lo 1 (p) \cap \lo 4 (p)|\over {p^{s_1 + s_4 + 1}}} + {|\lo 2 (p) \cap \lo 4 (p)|\over {p^{s_2 + s_4 + 1}}} \end{equation*}
\begin{equation*}
-  {|\lo 1(p) \cap \lo 2 (p) \cap \lo 3 (p)| \over {p^{s_1 + s_2 + s_3 +
1}}} - {|\lo 1(p) \cap \lo 2 (p) \cap \lo 4 (p)| \over {p^{s_1 + s_2 + s_4
+ 1}}} \\  - {|\lo 1(p) \cap \lo 3 (p) \cap \lo 4 (p)| \over {p^{s_1 + s_3
+ s_4 + 1}}}
\end{equation*}
\begin{equation*}
 - {|\lo 2(p) \cap \lo 3 (p) \cap \lo 4 (p)| \over {p^{s_2 +
s_3 + s_4 + 1}}} + {|\lo 1(p) \cap \lo 2 (p) \cap \lo 3 (p) \cap \lo 4
(p)| \over {p^{s_1 + s_2 + s_3 + s_4 + 1}}} \text{ \huge ) }.
\end{equation*}

In order to express this in a simpler fashion, define the following
function. Take $T \subset \{1,2,3,4\}$ and let $s(T) = \sum_{t \in T} s_t
+ 1$, $\Omega_T(p) = \bigcap_{t\in T} \Omega_t(p)$ and $E_T (p) =
|\Omega_T(p)| /  p^{s(T) }$. Rewriting the $F$ function,$$ \displaystyle
F(s_1, s_2, s_3, s_4) = \prod_p [1 + \hspace{-.2in} \sum_{\substack{T
\subset \{1,2,3,4\}
\\ T \neq \emptyset}} \hspace{-.2in}(-1)^{|T|} E_T(p) ].
$$
We now define our version of the $G$ function and prove a lemma on
its growth that we will need later. \\

$ G(s_1, s_2, s_3, s_4) =$

$ \displaystyle \hfill F(s_1, s_2, s_3, s_4) \left( { \zeta(s_1 +1) ^ {k_1} \zeta(s_2+1) ^ {k_2} \over
\zeta(s_1 + s_2 + 1) ^ {r_1}} \right) \left( { \zeta(s_3 +1) ^ {k_3}
\zeta(s_4+1) ^ {k_4} \over \zeta(s_3 + s_4 + 1) ^ {r_2}} \right).$ \vspace{.2in}

\noindent First, let $\HH = \HH_1 \cup \HH_2 \cup \HH_3 \cup \HH_4$ and, $
\displaystyle \Delta := \displaystyle \prod_{\substack{h_i, h_j \in \HH \\
h_i \neq h_j}} |h_j - h_i|$.\\ Choose $ \displaystyle U := C k^2 \log (h)$
so that $\displaystyle \log \Delta \leq U$.

\begin{lemma} \label{GLemma}
Let $\beta_i = \text{max}(-\Re(s_i), 0)$ and assume $\beta_i < 1/4 - \xi$ for all $0 < i \leq 4$ for some $\xi >
0$. Then, there exists a constant $C$ such that:
$$G(s_1, s_2, s_3, s_4) \ll_{M, \xi} \exp (C M U ^{\Sigma \beta_i}  \log \log
U). $$
\end{lemma}

\begin{proof} We will divide the total product into three separate
products and use the Euler product expansion of the zeta
function to bound each part of the product. \\

\subsubsection{The Product of Primes Under $U$}

We can divide the product into three parts, one which corresponds to terms of the $F$ function, one which
corresponds to the zeta functions in the numerator of $G$ and a third which corresponds to the zeta functions in
its denominator.

Bounding the product in the $F$ function for a fixed $T$ by standard means:

$ \displaystyle
 \prod_{p \leq U} \left(1 + \left| E_T(p)\right| \right)  \label{underu1}
\leq \prod_{p \leq U} \left(1 +  { \max(k_i) \over p^{1 - \Sigma
\beta_i}} \right)
\leq \exp\left( \sum_{p \leq U}{\max(k_i) \over p^{1- \Sigma \beta_i}} \right)$

$\hfill \leq \exp\left( \max(k_i) U^{\Sigma \beta_i} \displaystyle \sum_{p \leq
U}{1\over p} \right) \ll \exp(\max(k_i) U ^ {\Sigma{\beta_i}} \log \log U)$.

Since we can bound the total $F$ function by a fixed number of these
products, this total portion of the product is $\ll \exp(C \max(k_i) U ^
{\Sigma{\beta_i}} \log \log U)$.

Similarly to above, we invoke the results of \cite{goldston-2005} for
bounding the $\zeta$ functions in the numerator:
\begin{align*}
\prod_{p \leq U} \left| 1 - {1 \over p^{s_i + 1}} \right|^{-k_i} \ll
\exp(3 k_i U^{\beta_i} \log \log U).
\end{align*}

And, since $r_i$ is bound above by $\max_i(k_i)$ we can bound the $\zeta$
functions in the denominator as in \cite{goldston-2005},
\begin{align*}
\prod_{p \leq U} \left| 1 - {1 \over p^{s_i + s_j + 1}} \right|^{r_i}
&\leq \left( \prod_{p \leq U} \left( 1 + {1 \over p^{1-\Sigma \beta_i}}
\right) ^{-1} \right)^{r_i} \\
&\ll \exp( \max(k_i) U^{\Sigma \beta_i} \log \log U).
\end{align*}
Therefore, the final product for primes less than $U$ is
$$ \ll \exp(C
\max(k_i) U ^{\Sigma \beta_i} \log \log U).$$

\subsubsection{Primes above U which Divide $\Delta$}

For this, notice that,
\begin{align}
\prod_{\substack{p | \Delta \\ p > U}} \left(1 + {\max(k_i) \over p^{1 -
\Sigma \beta_i}} \right) \leq \exp \left( \displaystyle \sum_{\substack{p
| \Delta \\ p > U}} {\max{k_i} \over p^{1-\Sigma \beta_i}} \right).
\label{dividedelta1}
\end{align}

Similarly to the analysis in \cite{goldston-2005}, there are less than $(1
+ o(1))\log \Delta < U$ primes such that $p | \Delta$. We can therefore
replace the sum above with the first $U$ numbers greater than $U$.
Therefore,
\begin{equation*}
(\ref{dividedelta1}) \leq \exp \left( \max(k_i) \sum_{U < n \leq 2U}{1
\over n^{1 - \Sigma \beta_i}}\right) \leq \exp(C \max(k_i) U^{\Sigma
\beta_i}).
\end{equation*}

And each of the factors can be bound identically. \subsubsection{Primes
above $U$ which do not divide $\Delta$}

At this point the choice of the zeta functions and our restriction on the
intersection of the sets becomes important. By their selection, the size
of the intersections of the $\Omega_i(p)$ functions are exactly known in
this category. If two of the values occupied the same residue class modulo
a given prime, then their difference would have been a factor in computing
$\Delta$ and therefore, the prime would have divided $\Delta$, and
therefore is not in this product. So for all $p$ such that $p \nmid
\Delta$,
\begin{equation*} |\Omega_i| = k_i, \; |\Omega_1 \cap \Omega_2| = r_1 , \;|\Omega_3
\cap \Omega_4| = r_2
\end{equation*}
and all other combinations are empty by our initial assumption that
\begin{equation*} (\HH_1 \cup \HH_2) \cap (\HH_3 \cup \HH_4) = \emptyset.
\end{equation*}
In this case, we note the following about the product of terms belonging
to the $F$ function. Assuming $U$ is bigger than a certain threshold
depending only on $\max(k_i)$ and $\xi$, \\

$\noindent \displaystyle \prod_{\substack{p \nmid \Delta \\ p > U}} [1 + \hspace{-.1in}\sum_{\substack{T \subset \{1,2,3,4\} \\
T \neq \emptyset}} \hspace{-.2in}(-1)^{|T|} E_T(p) ] =$

$\noindent  \displaystyle \prod_{\substack{p \nmid \Delta \\ p > U}}
\left( {1 - {k_1 \over p^{1 + s_1}} - {k_2 \over p^{1 + s_2}} + {r_1 \over
p^{1 + s_1 + s_2}} - {k_3 \over p^{1 + s_3}} - {k_4 \over p^{1 + s_4}} +
{r_2 \over p^{1 + s_3 + s_4}}} \right) \leq $

 $\noindent  \displaystyle \prod_{\substack{p \nmid \Delta \\
p > U}} \left| \left(1- {k_1 \over p^{1+s_1}}\right) \left(1- {k_2 \over
p^{1+s_2}}\right) \left(1+ {r_1 \over p^{1+s_1 + s_2}}\right) \left(1-
{k_3 \over p^{1+s_3}}\right) \right. \times $ \vspace{-.1in}

$\hfill \displaystyle \left. \left(1- {k_4 \over p^{1+s_4}}\right)
\left(1+ {r_2 \over p^{1+s_3 + s_4}}\right) \left( 1 + {C\max(k_i) \over
p^{ 2 - 4 \max {\beta_i}}} \right) \right| \ll_{M,\xi} $ \vspace{.1in}

$\noindent\displaystyle\prod_{\substack{p \nmid \Delta \\ p > U}} \left|
\left(1- {k_1 \over p^{1+s_1}}\right) \left(1- {k_2 \over
p^{1+s_2}}\right) \left(1+ {r_1 \over p^{1+s_1 + s_2}}\right) \left(1-
{k_3 \over p^{1+s_3}}\right) \times \right.$ \\ \vspace{-.1in}

$\hfill \displaystyle \left. \left(1- {k_4 \over p^{1+s_4}}\right)
\left(1+ {r_2 \over p^{1+s_3 + s_4}}\right) \right|. $ \vspace{.1in}

Where the final inequality follows from the fact that $\beta_i < 1/4 -
\xi$. Now each of these factors can be paired with its corresponding zeta
functions. There are two distinct such pairs, and only two lines of
analysis are necessary. Using the following results from
\cite{goldston-2005}:

\begin{align*}
\prod_{\substack{p \nmid \Delta \\ p > U}} \left|  \left(1- {k_i \over
p^{1+s_i}}\right)\left(1- {1 \over p^{1+s_i}}\right)^{-k_i} \right| \leq
\exp(2k_iU^{\beta_i}), \end{align*} \begin{align*}
 \prod_{\substack{p \nmid \Delta \\ p > U}} \left| \left(1+ {r_i \over p^{1+s_i+s_j}}\right)\left(1- {1 \over
p^{1+s_i+s_j}}\right)^{r_i} \right| = \prod_{\substack{p \nmid \Delta \\ p
> U}} \left( 1 + O_M\left({1 \over p^{2 - 2s_i - 2s_j}} \right) \right)
\end{align*}

\noindent where the product on the right hand side of the second inequality is convergent depending only on
$\xi$ due to our assumption that $\beta_i < 1/4 - \xi$. Therefore, the entire product for primes in this
category is $\ll \exp(C \max(k_i) U^{\Sigma \beta_i})$. Therefore, combining the results of all three sections,
the total product is,
\begin{equation*}
\ll_{M, \xi} \exp(C M  U ^{\Sigma \beta_i} \log \log U)
\end{equation*}

\noindent for some $C > 0$ which does not depend on any $k_i$ value. This growth condition is necessary to
invoke the result of \cite{goldston-2005}. It should also be noted that as in previous works, $G$ is analytic in
the region described in this lemma as this property will be needed later (this follows from the definition of
$G$ and the bound we just exhibited). We now give a few results on the zeta function cited in
\cite{goldston-2005} before stating the necessary lemma.
\end{proof}

\subsection{Some Facts on the Zero Free Region of $\zeta$}

First, there is a small constant $\bar{c}\leq 10^{-2}$ such that
$\zeta(\sigma + it) \neq 0$ in the region,

$$ \sigma \geq 1 - {4 \bar{c} \over \log(|t| + 3)}. \;\; \text{ Furthermore, in this region: }$$

$$\zeta(\sigma + it) - {1 \over \sigma - 1 + it} \ll \log(|t| + 3),
\; \text{ and } \;\; {1 \over \zeta(\sigma + it)} \ll \log(|t|-3).$$

\subsection{A Necessary Lemma} \label{GPYSection}

In the analysis presented in this paper, only a weaker version of the
lemma in \cite{goldston-2005} is needed. The version needed is stated
below. Let,
$$\displaystyle T_R^*(a,b,d,u,v,h) := { 1 \over (2 \pi i)^2}
\int_{(1)} \int_{(1)} {D(s_1, s_2)R^{s_1 + s_2} \over s_1^{u+1} s_2^{v+1}
(s_1 + s_2)^d} ds_1 ds_2$$ \noindent where,
$$D(s_1, s_2) := {G(s_1, s_2)W^d(s_1 + s_2) \over W^a(s_1)
W^b(s_2)} \text{ and } W(s) := s \zeta ( 1+s).$$

Assume $G(s_1, s_2)$ is regular on and to the right of the line:
\begin{equation}
s = -{\bar{c} \over \log(|t|+3)} + it
\end{equation}
and satisfies the bound:
$$G(s_1, s_2) \ll_M \exp(CMU^{ \beta_1 + \beta_2} \log \log U),
\text{ where  }U = C M^2 \log (2h).$$
\begin{lemma} \label{GPYLemma} Suppose that
$$0 \leq a, b, d, u, v \leq M, \text{ } a + u \geq 1, \text{ } b + v
\geq 1, \text{ } d \leq \min(a, b)$$

\noindent where $M$ is any large constant. Let $h \ll R^C$ for any $C >
0$. Then, as $R \to \infty$,

$\displaystyle T_R^*(a, b, d, u, v, h) = {u + v \choose u} { (\log R)^{u + v + d}
\over (u + v + d)!} G (0, 0) +$

$\hfill \displaystyle O_M((\log R)^{u+v+d-1} (\log \log
R)^{C_M}).$

\end{lemma}

\subsection{Splitting the four integrals} \label{SplittingSection}

The goal in this section will be to split the four integrals in the
expression for $T$ into two pairs of two integrals and use {\it Lemma
\ref{GPYLemma}} twice. We will use {\it Lemma \ref{GLemma}} to show that
each of the pairs of integrals is acceptable to use with {\it Lemma
\ref{GPYLemma}}. Let $k_i + l_i +1 = u_i$ and introduce the following
notation (let $s_j = \sigma_j + i t_j$):
$$  {{ \zeta(s_1 + s_2 +
1) ^ {r_1}}  \over { \zeta(s_1 +1) ^ {k_1} \zeta(s_2+1) ^ {k_2}}}
 = \zeta_1(s_1, s_2) \;\;\;\; , \;\;\;\; {{ \zeta(s_3 + s_4 + 1) ^ {r_2}}  \over { \zeta(s_3
+1) ^ {k_3} \zeta(s_4+1) ^ {k_4}} } = \zeta_2 (s_3, s_4).$$ This allows a
simplification of the expression for $T$ as (let $\overline{d} = ds_1 ds_2 ds_3 ds_4$):
\begin{align*}
T &=& \int_{(1)} \int_{(1)} \int_{(1)} \int_{(1)} G(s_1, s_2, s_3, s_4)
\zeta_1(s_1, s_2) \zeta_2(s_3, s_4) {R^{s_1 + s_2 + s_3 + s_4} \over
s_1^{u_1} s_2 ^ {u_2} s_3 ^ {u_3} s_4 ^ {u_4}} \overline{d} .
\end{align*}
The integrand above in $s_3$ is analytic to the right of the line $\Re(z)
= 0$ as long as $\Re(s_i) > 0$ for all other $s_i$ and the same holds true
for $s_4$. Checking {\it Lemma \ref{GLemma}}, one sees that the integrand
in $T$ vanishes as either $|t_3| \to \infty$ or $|t_4| \to \infty$. One
can therefore shift the integral over both variables to the line $L$ which
is the vertical line which passes through $1 / \log(N)$. Therefore, with a
quick substitution,

$$T ={1 \over (2 \pi i)^2} \int_{(L)} \int_{(L)} Q(s_3, s_4) \zeta_2(s_3, s_4) {R^{s_3 + s_4} \over
s_3^{u_3} s_4^{u_4}} ds_3 ds_4 \;\;\;\;\;\;\;\;\;\; \text{where, }$$

$$Q(s_3, s_4) := {1 \over (2 \pi i)^2}\int_{(1)} \int_{(1)} G(s_1, s_2, s_3,
s_4) \zeta_1(s_1, s_2) {R^{s_1 + s_2} \over s_1^{u_1} s_2 ^ {u_2} } ds_1
ds_2.$$

\subsection{Applying the Lemma}

By {\it Lemma \ref{GLemma} }when $s_3, s_4$ are fixed on the line
$\Re(s_3) = \Re(s_4) = 1/\log(N)$ (Which implies $\beta_3 = \beta_4 = 0$)
and taking $\xi$ as a small enough absolute constant depending only on the
absolute constant $\bar{c}$ to ensure the region $\beta_i \leq 1/4 - \xi$
includes the zero free region described previously,
$$G(s_1, s_2, s_3, s_4) \ll_M \exp(C M \log U ^{\beta_1 + \beta_2} \log \log
U)$$ \noindent to the right of the line described before {\it Lemma
\ref{GPYLemma}}. Therefore the $G$ function is acceptable to use {\it
Lemma \ref{GPYLemma}} where $s_3$ and $s_4$ are fixed and positive and the
integral is being evaluated over the $s_1$ and $s_2$ variables. Using the
substitution with the $W$ function defined identically as in {\it Lemma
\ref{GPYLemma}},

\begin{equation*}
 D_{s_3, s_4}(s_1, s_2) = {G(s_1,s_2,s_3,s_4) W^{r_1}(s_1, s_2)
\over W^{k_1}(s_1)W^{k_2}(s_2)},
\end{equation*}

\begin{equation*}
Q(s_3, s_4) = {1 \over (2 \pi i)^2} \int_{(1)} \int_{(1)} D_{s_3,
s_4}(s_1, s_2) {R^{s_1 + s_2} \over s_1^{l_1+1} s_2 ^ {l_2+1}(s_1 +
s_2)^{r_1} } ds_1 ds_2.
\end{equation*}

Because the analyticity of $G$ in the $s_1$ and $s_2$ variables is
maintained when $s_3$ and $s_4$ are fixed and positive, the lemma can be
applied by letting $a = k_1$, $b = k_2$, $u = l_1$, $v = l_2$ and $d =
r_1$. As long as $k_1 + l_ 1 \geq 1$ and $k_2 + l_2 \geq 1$ (for the rest
of the section, assume that each $k_i$ is positive, implying the previous
inequality. The case when at least one of the $k_i = 0$ will be addressed
separately at the end of the section). {\it Lemma \ref{GPYLemma}} implies,
as long as $\Re(s_3), \Re(s_4) \geq 0$, \vspace{.2in}

$\displaystyle Q(s_3, s_4) = {l_1 + l_2 \choose l_1}{(\log R)^{l_1 + l_2 + r_1}
\over {(l_1 + l_2 + r_1)}!}G(0, 0, s_3, s_4) + $

$\hfill \displaystyle O_M \left( (\log N)^{l_1 +
l_2 + r_1 -1} (\log \log  N) ^ {C_M} \right)$.

 This implies, \\

$\displaystyle T = {l_1 + l_2 \choose l_1}{(\log R)^{l_1 + l_2 + r_1}
\over {(l_1 + l_2 + r_1)}!}\int_{(L)} \int_{(L)} G(0, 0, s_3, s_4)
\zeta_2(s_3, s_4) {R^{s_3 + s_4} \over s_3^{u_3} s_4^{u_4}} ds_3 ds_4
$

$\displaystyle \hfill + O  \left( (\log N)^{l_1 + l_2 + r_1 -1} (\log \log
N) \right) \int_{(L)} \int_{(L)} \left| \zeta_2(s_3, s_4) {R^{s_3 + s_4}
\over s_3^{u_3} s_4^{u_4}} \right|ds_3 ds_4. $\\

Since the first integrand vanishes as $t_1 \to \infty$ or $t_2 \to \infty$, one can shift the lines of
integration of the first line above back to the line $\Re(z) = 1$, there are therefore two values left to
evaluate:
\begin{align*} T_1 = \int_{(1)} \int_{(1)} G(0, 0, s_3, s_4) \zeta_2(s_3, s_4)
{R^{s_3 + s_4} \over s_3^{u_3} s_4^{u_4}} ds_3 ds_4 \;\;\; \text{ and, }
\;\;\;
\end{align*}

$$T_2 = \int_{(L)} \int_{(L)} \left| \zeta_2(s_3, s_4) {R^{s_3 + s_4} \over
s_3^{u_3} s_4^{u_4}}\right| ds_3 ds_4  $$

\subsection{Evaluating $T_1$, the main term and some error} \label{T1sub}

Since the roles of $(s_1, s_2)$ and $(s_3, s_4)$ are symmetric in {\it
Lemma \ref{GLemma}} (when $s_1= s_2 = 0$, $\beta_1 = \beta_2 = 0$), one
can switch the roles of $s_1, s_2$ with $s_3, s_4$ and reapply {\it Lemma
\ref{GPYLemma}} to get,

$\displaystyle T_1 = {l_3 + l_4 \choose l_3}{(\log R)^{l_3 + l_4 + r_2} \over {(l_3 +
l_4 + r_2)}!}G(0, 0, 0, 0) +$

$\hfill O_M\left( (\log N)^{l_3 + l_3 + r_2 -1} (\log
\log N)^{C_M} \right).$

\subsection{Evaluating $T_2$, the second error term} \label{T2sub}
\label{seconderror}

Fix any absolute constant $\gamma \in (0, 1)$. Noticing that $R^{s_3 + s_4}$ is absolutely bounded on the line
$1/\log N$:
\begin{align*}
T_2 &\ll \int_{(L)} \int_{(L)} \left| {\zeta(s_3+ s_4+1)^{r_2} \over
\zeta(s_3 + 1)^{k_3} \zeta(s_4 + 1)^{k_4}} {1 \over s_3^{k_3+l_3+ 1}
s_4^{k_4+l_4 + 1}}   \right| ds_3 ds_4.
\end{align*}


Now observe that with $k_3 \geq 1$, when $\Re(s_3) \geq 0$ and $|s_3| \leq
1$,

$$\left| {1 \over \zeta(s_3 + 1)^{k_3} s_3^{k_3}} \right| \ll_M 1 \ll_M { 1 \over
  \left| {s_3}^\gamma \right|},$$

\noindent and when $|s_3|
> 1$, there is the general inequality that follows from the growth
conditions enumerated in {\it Section \ref{ZetaSection}} (since if
$\Re(s_3) \geq 0$, $s_3+1$ trivially falls in the region described),

$$\left| {1 \over \zeta(s_3 + 1)^{k_3} s_3^{k_3}} \right| \ll_M \left| {1 \over
    \zeta(s_3 + 1)^{k_3} {s_3}} \right| \ll_M {\log(|t_3| + 3)^{k_3}
  \over \left| {s_3} ^ {\gamma} \right|},$$

\noindent and finally, the inequality,
$$\left| \zeta(s_3 + s_4 + 1)^{r_2} \right| \ll_M  \log(|t_3 + t_4| + 3)^{r_2} \max\left( 1, {1 \over |s_3 + s_4 |}
\right)^{r_2}.$$

Substituting $\omega_3 = x_3 + i y_3 = s_3 \log(N)$ and letting $\overline{d} = d s_3 d s_4, \hat{d} = d \omega_3 \omega_4$,\\

$\displaystyle T_2 \ll_M \int_{(L)} \int_{(L)} \hspace{-.1in}{\log(|t_3| + 3)^{k_3} \log(|t_4| + 3)^{k_4} \log(|t_3 + t_4|
+ 3)^{r_2}  \over \left|s_3^{l_3+ 1 + \gamma} s_4^{l_4 + 1 + {\gamma}}\right|} \times $\\

$\hfill \displaystyle \max\left( 1,
{1 \over |s_3 + s_4 |} \right)^{r_2}  \overline{d} \ll_M $\vspace{.1in} \\

$\displaystyle \log(N)^{r_2} \int_{(L)} \int_{(L)}\hspace{-.1in} {\log(|t_3| +  3)^{k_3} \log(|t_4| +
3)^{k_4} \log(|t_3 + t_4| + 3)^{r_2} \over
  \left|s_3^{l_3+ 1 + \gamma} s_4^{l_4 + 1 + \gamma}\right|} \overline{d}
  \ll_M$\vspace{.1in} \\

$\displaystyle \hspace{-.15in} \log(N)^{l_3 + l_4 + r_2 + 2 \gamma } \times$

$\displaystyle \hfill \int_{(1)} \int_{(1)}
{\log(|{y_3 \over \log(N)}| +  3)^{k_3} \log(|{y_4 \over \log(N)}| + 3)^{k_4} \log(|{y_3 + y_4 \over \log(N)}| +
3)^{r_2} \over {|\omega_3 ^{l_3 + 1 + \gamma} \omega_4^{l_4 + 1 +
\gamma}}|} \hat{d} \ll_{M,\gamma} \vspace{.1in} $\\

$$\displaystyle \log(N)^{l_3 + l_4 + r_2 + 2 \gamma}$$
\noindent because the final integral is absolutely convergent since $l_i
\geq 0$ and $\gamma > 0$.

\subsection{Combining the results}

This yields the final derivation that, using the fact that $G(0,0,0,0) =
\gs (\HH_1 \cup \HH_2 \cup \HH_3 \cup \HH_4)$ and labeling $l_1 + l_2 +
r_1 = u$ and $l_3 + l_4 + r_2 = v$:

$\displaystyle T ={l_1 + l_2 \choose l_1}{l_3 + l_4 \choose l_3} {(\log R)^{u + v} \over {u!v}!} \gs(\HH_1 \cup \HH_2 \cup \HH_3 \cup \HH_4) +$

$\displaystyle \hfill O_{M,\gamma} \left( (\log N)^{u +
v - 1 + 3\gamma} \right)$

\noindent Since we can pick $\gamma$ as any positive value, this implies
the theorem when combined with the additional error term from {\it Section
\ref{TheFirstError}}. Now, let us address the case where some $k_i = 0$.
If two $k_i$ values are zero, the theorem is implied by the result from
{\it Proposition 1} from \cite{goldston-2005} because there are only two
remaining weight functions. The remaining case is when only one $k_i$
value is zero. In this case, instead of $4$ integrals, there are only
three remaining. The analysis up to {\it Section \ref{T1sub}} is
identical, with the only change being that there are three integrals
instead of four. At {\it Section \ref{T1sub}} instead of invoking {\it
Lemma \ref{GLemma}}, one would use the analysis of {\it Proposition 1
(Special Case)} of \cite{goldston-2005}, which is the equivalent statement
of {\it Proposition 1} with only one weight function instead of two (it is
only explicitly shown for $l=0$ but as is mentioned in {\it Section 6} of
\cite{goldston-2005} the analysis generalizes to all $l \geq 0$). The
corresponding analysis of {\it Section \ref{T2sub}} follows identically
with one integral instead of two.

\section{Proof of Theorem 2, Gallagher Extension} \label{GallagherSection}

In this section, we will show that Kevin Ford's \cite{ford-2007}
simplification of P. X. Gallagher's proof can be extended to give an
estimate for a more involved sum over singular series which is useful is
applying {\it Theorem 1}. The importance of the theorem is that it allows
the two pairs of weights we consider to vary over different intervals.
\begin{thm2}

Take an interval $[0, h]$ and take $l$ subintervals $[B_i(h), C_i(h)] \subset [0, h]$ where $C_i(h) - B_i(h) =
d_i (h)$. Assume for some $1>\delta > 0$, and for all $i \leq l$, $h^\delta = o(d_i (h))$. Then,
\begin{equation*} \label{originaleq}
 \sum_{\substack{A_1, A_2, \ldots A_l: \\ A_i \subset [B_i(h), C_i(h)], |A_i| = k_i}} \gs \left( \bigcup_{i=1}^l A_i \right) = \prod_{i = 1}^l {d_i^{\;
k_i}  (h) \over k_i !} \left( 1 + O_{r, \delta}\left({1 \over \log \log h } \right)
\right).
\end{equation*}
\end{thm2}

First, recall the definition of the singular series:
$$ \gs( \HH ) = \prod_{p} \left( 1 -
{|\Omega_\HH(p)| \over p} \right) \left(1 - { 1 \over p} \right) ^{-
|\HH|}.$$

 Similarly to the proof of {\it Lemma \ref{GLemma}} define, with $\HH = \cup
 A_i$,
$$\Delta := \prod_{h \neq h' \in \mathcal{H}} |h - h'| \text{ , }y
:= (\delta / 2) \log h \text{ and } r:= \sum_{0 < i \leq l} k_i. $$

The first statement to show, will be that it suffices to consider all $A_i$ such that $i \neq j \Rightarrow A_i
\cap A_j = \emptyset$ because the number of such sets vastly exceed all others. The singular series itself can
be bound above without much trouble, notice that the product for all $p > h$ is $\ll_r 1$ since the size of
$|\HH|$ and $|\Omega_{\HH}(p)|$ will both be equal. As for the product of primes under $h$, we can bound the
product as $\ll \log^r h$ by Mertens' Theorem. Since $h^{\delta}$ grows more slowly than any $d_i (h)$,
\begin{equation} \label{originaleq2} \hspace{-.3in}
 \sum_{\substack{A_1, A_2, \ldots A_l: \\ A_i \subset [B_i(h), C_i(h)], |A_i| = k_i}} \hspace{-.2in}\gs ( \bigcup_{i=1}^l A_i ) =
\hspace{-.2in} \sum_{\substack{A_1, A_2, \ldots A_l: \\ A_i \subset [B_i(h), C_i(h)]
 \\|A_i| = k_i, | \cup A_i | = r}} \hspace{-.1in} \gs ( \bigcup_{i = 1}^l A_i )
 + O_{r}(h^{-\delta} \log^r h\prod_{i=1}^l d_i^{k_i} (h) ).
\end{equation}

The sum on the right hand side is easier to evaluate because the exponent in the definition of the singular
series will now be a constant $-r$. Now, fix any $A_1, A_2, \ldots A_l$ which fall within the subintervals such
that $| \cup A_i | = r$, it is shown in \cite{ford-2007} that,

\begin{equation*}
\prod_{p > y} \left( 1 - {|\Omega_{\mathcal{H}}(p)| \over p} \right)
\left(1 - {1 \over p} \right)^{-r} = 1 + O_{r, \delta} \left( 1 \over \log
\log h \right).
\end{equation*}

As such, the sum on the right hand side of (\ref{originaleq2}) is equal to, \vspace{-.1in}

\begin{equation}\label{eqgal}\left(1 + O_{r,\delta} \left(1 \over \log \log h \right) \right) \prod_{p \leq
y} \left( 1 - {1 \over p}  \right) ^{-r} \hspace{-.1in}  \times \hspace{-.2in} \sum_{
\substack{A_1, A_2, \ldots A_l: \\ A_i \subset [B_i(h), C_i(h)]  \\|A_i| =
k_i, | \cup A_i | = r} } \prod_{p \leq y} \left( 1 - {
|\Omega_{\mathcal{H}}(p)| \over p } \right). \end{equation}

Let $ P = \prod_{p \leq y} p$ and note that $P = e^{y + o(y)} = h^{(\delta/2) + o(1)}$. The product on the far
right is $1/P$ times the number of $n$, $0 \leq n < P$ such that $ \left( \prod_{\alpha \in \mathcal{H}} (n+
\alpha), P \right) = 1$. We can also now eliminate the $|\cup A_i| = r$ condition with an error term $O_r(
h^{-\delta}\prod_i d_i^{k_i} (h))$, this leaves the factor on the right of (\ref{eqgal}) as, \vspace{-.15in}

\begin{align*} & \sum_{\substack{A_1, A_2, \ldots A_l: \\ A_i \subset [B_i(h), C_i(h)]
 \\|A_i| = k_i} } \hspace{-.15in}{ 1 \over P} \sum_{n=0}^{P-1}
\prod_{\alpha \in \HH} \sum_{e | (n+ \alpha, P)} \hspace{-.15in} \mu(e)  + O_r(h^{-\delta} \prod_{i=1}^l
d_i^{k_i} (h) )= \end{align*} \vspace{-.15in}
\begin{equation*} { 1 \over P} \sum_{n=0}^{P-1} \hspace{-.14in}\sum_{\substack{A_1, A_2, \ldots A_l: \\
A_i \subset [B_i(h), C_i(h)] \\ |A_i| = k_i} } \hspace{-.04in} \prod_{\alpha \in \HH} \sum_{e | (n+ \alpha, P)}
\hspace{-.14in} \mu(e) + O_r(h^{-\delta} \prod_{i=1}^l d_i^{k_i} (h) ). \end{equation*}

Let $Q(\alpha, e_i) = 1$ if $e_i | n + \alpha$ and $0$ otherwise. Then , letting $\HH' = \{ \alpha_1, \alpha_2,
\ldots \alpha_r \}$ where the first $k_1$ elements are in $[B_1(h), C_1(h)]$, the next $k_2$ are from $[B_2(h),
C_2(h)]$ and so on, with the last $k_l$ being from the interval $[B_l(h), C_l(h)]$,
$$ \sum_{\substack{A_1, A_2, \ldots A_l: \\
A_i \subset [B_i(h), C_i(h)]  \\|A_i| = k_i }} \prod_{ \alpha \in \HH} \sum_{e | (n+ \alpha, P)} \mu(e) = \left[
\prod_{i = 1} ^ l {1 \over k_i!} \right]  \sum_{\HH'} \prod_{ \alpha \in \HH'} \sum_{e | (n+ \alpha, P)} \mu(e)
=$$ $$ \left[ \prod_{i = 1}^l {1 \over k_i!} \right] \sum_{\HH'}  \sum_{e_1, e_2, \ldots, e_r | P} \mu(e_1)
\mu(e_2) \ldots \mu(e_r)
 \left[ \prod_{j = 1}^r Q(\alpha_j, e_j)  \right] $$

\noindent and therefore, the right side of (\ref{eqgal}) is, apart from the error term, equal to: \vspace{-.1in}

\begin{equation*}
{ 1 \over P} \sum_{n=0}^{P-1} \left[ \prod_{i = 1}^l {1 \over k_i!}
\right]  \sum_{e_1, e_2, \ldots, e_r | P} \mu(e_1) \mu(e_2) \ldots
\mu(e_r)
 \left[  \sum_{\HH'} \prod_{j = 1}^r Q(\alpha_j, e_j)  \right].
\end{equation*}

For a fixed $e_1, e_2, \ldots e_r$, $Q(\alpha_i, e_i)$ will be $1$ a total of $d_j (h) / e_i + O(1)$ times over all
choices of $a_i$ from $[B_j(h), C_j(h)]$ independent of all other $a_j$. By the definition, exactly $k_j$ of the
$a_i$ were chosen from $[B_j(h), C_j(h)]$. Since $P$ is an upper bound for each $e_i$ and $h^\delta$ grows much
slower than any $d_i(h)$, it follows that,
\begin{equation*}
\sum_{\HH'} \prod_{j = 1}^r Q(\alpha_j, e_j) = {d_1 ^ {k_1} (h) d_2 ^
{k_3} (h) \ldots d_l ^ {k_l} (h) \over e_1 e_2 \ldots e_r }\left( 1 +
O_{r}(h^{- \delta}{ P }) \right).
\end{equation*}

We have now eliminated the dependence on $n$ and therefore the $P$ and
$\sum_{n = 0} ^ {P-1}$ cancel out, leaving the right side of (\ref{eqgal})
as,

$\displaystyle
\left[ \prod_{i = 1}^l {d_i ^ {k_i} (h) \over k_i!} \right] \sum_{e_1, e_2, \ldots, e_r | P} { \mu(e_1) \mu(e_2)
\ldots \mu(e_r) \over e_1 e_2 \ldots e_r} \left( 1 +  O_{r}(h^{- \delta}{ P }) \right) + \vspace{-.2in}$

$\displaystyle \hfill  O_r(h^{-\delta}
\prod_{i=1}^l d_i^{k_i} (h) ).
$

Which, after substituting in the previous bound on the growth of $P$ in
terms of $h$, leaves the desired:

$ \displaystyle \left[ \prod_{i = 1}^l {d_i ^ {k_i} (h) \over k_i!} \right] \sum_{e_1,
e_2, \ldots, e_r | P} { \mu(e_1) \mu(e_2) \ldots \mu(e_r) \over e_1 e_2 \ldots e_r} \left( 1 + O_{r}(h^{-
\delta/2 + o(1)}) \right) +\vspace{-.2in}$

$\displaystyle \hfill O_r(h^{-\delta} \prod_{i=1}^l d_i^{k_i} (h) ) =
$ \begin{equation*} \left[ \prod_{i = 1}^l {d_i ^ {k_i} (h) \over k_i!}
\right] \prod_{p \leq y} \left( 1 - {1 \over p} \right) ^ r \left( 1 + O_{r}(h^{- \delta/2 + o(1)}) \right) +
O_r(h^{-\delta} \prod_{i=1}^l d_i^{k_i} (h) ).\end{equation*}

Multiplying the above with the left side of (\ref{eqgal}) and incorporating the error term from
(\ref{originaleq2}) yields the original sum as equal: \vspace{-.15in}

\begin{equation*}
\prod_{i = 1}^l {d_i^{\; k_i}  (h) \over k_i !} \left( 1 + O_{r, \delta}\left(1 \over \log \log h \right) \right) + \prod_{p \leq
y} \left( 1 - {1 \over p}  \right) ^{-r} \hspace{-.1in}O_{r, \delta}(h^{-\delta/2 + o(1)}\prod_{i=1}^l d_i^{k_i} (h)) .
\end{equation*}

Using Mertens' Theorem again to bound the product for $p \leq y$ on the right as $\ll_r \log^r y$ allows us to
incorporate the second error term with the first which implies the theorem.

\section{Application to $F_2$}

In this section, we will show an example application of the previous two theorems. We will use them to show an
improvement over the result for $F_2$ obtained in \cite{goldston-2005}.

\begin{thm3} There exists a $c > 0$ such that $F_2 \leq c < (\sqrt{2} - 1)^2$ . \end{thm3}

Following Goldston, Pintz and Y{\i}ld{\i}r{\i}m's lead, some sets of size
$k$ will be counted with multiplicity of $k!$ according to their
permutations. If a subset is meant to be taken with multiplicity in this
fashion, we will use the notation $\subset^*$. The way Goldston, Pintz and
Y{\i}ld{\i}r{\i}m proved their result for $\Delta_2$ is through the
following method. They showed that assuming $h > (\sqrt{2} - 1) ^2 \log
(N)$, the difference
\begin{equation} \label{GPYDif}
A( \bv ) := \sum_{n = N+1}^{2N} [ \sum_{1 \leq h_0 \leq h} \theta (n +
h_0) - 2 \log (3N) ] [ \hspace{-.2in}\sum _ {\substack{\HH \subset^* \{1
,2 \ldots , h \}
\\ { | \HH | = k}}}\hspace{-.2in} \Lambda _ R ( n ; \HH, k+l) ] ^2
\vspace{-.1in}
\end{equation}
\begin{align*}\text{where, }  R = N^{\Theta}, \; \; h = \lambda \log (3N) \;\; \text{and} \;
\; \bv =  \lambda, \Theta, k,l,N
\end{align*}
\noindent is positive, which implies there are three primes in the
interval $n + H := n + \{1, 2, \ldots, h\}$ ($H = [1,h]$) for some $n \in
[N+1, \ldots, 2N]$. For our derivation, we will modify their analysis in
the following way. Instead of considering one, we will consider three
intervals, where $h' = \delta h$, $\delta < 1/2$:

$$H_1 := \{i \in \mathbb{Z} : 0<i \leq h'\}, \; H_2 := \{i \in \mathbb{Z} : h'<i < h-h'\}, $$ $$H_3 := \{i \in \mathbb{Z} : h-h'\leq i \leq h\}.$$

If we could modify the difference $A( \bv )$ in such a way that positivity
not only implied that there are three primes in the interval $n + H$ but
also that one of these primes came from $n + H_2$ we could then guarantee
that $F_2 \leq (h - h')/ \log(3N)= \lambda (1 - \delta)$ whenever the
modified difference is positive. If no primes come from the central
interval and there are in total less than $3$ primes, the difference
(\ref{GPYDif}) is already negative. Moreover, if three primes lie in $n +
H_1$ or three primes lie in $n + H_3$, it would imply $F_2 \leq h'$, which
would imply $F_2 < \lambda (1- \delta)$ since $\delta < 1/2$. If there are
$5$ or more primes, either three primes come from an end interval or there
is one in the central interval. With a little consideration, we can see
that in order for positivity to imply our bound on $F_2$ all we need to do
is add another negative term to the difference (\ref{GPYDif}) which would
assure that the following `bad cases' also lead to the difference not
being positive:

\begin{enumerate}
\item Two primes in $n + H_1$, one in $n + H_3$ and no other in $n + H$
\item One prime in $n + H_1$, two in $n + H_3$ and no other in $n + H$
\item Two primes both in $n + H_1$ and in $n + H_3$ and no other in $n + H$.
\end{enumerate}

If the added term made these three cases negative as well, we could
guarantee that $F_2 \leq \lambda(1 - \delta)$. The proof strategy now
relies on the fact that the number of triples of primes coming from these
very short intervals ($H_1$ and $H_3$), should be approximately
proportionate to $\delta^3$ times the number of triples of primes coming
from the whole interval (if we assign each number $n$ an independent
probability of $1/\log n$ of being prime). As we expand the interval by a
factor of $1 + \delta$, the positive contribution from the first term in
$(\ref{GPYDif})$ grows linearly with respect to this factor. The negative
contribution adds substantially less to the overall sum than the positive
term increases when $\delta$ is small. This provides the leverage we need to lower the bound for $F_2$ even though our total interval is bigger than the one used in \cite{goldston-2005}'s proof for $\Delta_2$. \\

Consider the following term:
\begin{equation*} \label{sum1} B_1(\bv, \delta) := \displaystyle {\sum_{n = N} ^ {2N}}^*  \left( \sum_{|\HH|=k} \LH {}
\right)^2 \log(3N). \end{equation*} \noindent Where the starred summation
indicates that the term $n$ is only counted if the interval $n + H_1$ has
exactly two primes and $n + H_3$ has at least one prime. Let $B_2$ be the
same summation where $n+H_3$ has exactly two primes and $n + H_1$ has at
least one. One can see that subtracting $B_1 (\bv, \delta) + B_2 (\bv,
\delta)$ would satisfy the requirements that all three cases above would
not contribute positivity to the overall sum. So, if we can show for a
given choice of $\lambda$, $\delta$, that there exists $\Theta < 1/4$ and $k,l
\in \mathbb{N}$ such that $A(\bv) - B_1(\bv, \delta) - B_2 (\bv, \delta) >
0 $ it will imply $F_2 \leq \lambda(1 - \delta)$.

Now, we use a second set of weights in order to find an upper bound for
$B_1 (\bv, \delta)$, the analysis can be repeated identically for
$B_2(\bv, \delta)$: \vspace{.1in}

 $ \displaystyle B_1(\bv, \delta) \leq$

$\hfill \displaystyle \Omega_1 {\sum_{n
= N} ^{2N}} \sum_{\substack{A_1 \subset H_1, A_2 \subset H_3 \\ |A_1| = 2, |A_2| = 1}} \hspace{-.2in}
\Lambda_R(n, A_1 \cup A_2 ,3)^2  [ \sum_{\substack{\HH \subset^* H \\ |\HH|=k }} \Lambda_R(n,\HH \setminus (A_1
\cup A_2),k+l) ]^2  $ \vspace{.1in}

\noindent where, $ \Omega_1 := \left({ 36 \log(3N) / \log^6(R)} \right)$.
This bound holds because if for all $a \in (A_1 \cup A_2)$, $n + a$ is
prime,
\begin{equation*}
\Lambda_R(n,A_1 \cup A_2,3)^2 =  {\log^6(R) \over 36} \;\;\;\;\; \text{
and, }
\end{equation*}
\begin{equation*}
\Lambda_R(n,\HH ,k+l) =  \sum_{|\HH|=k} \Lambda_R(n,\HH \setminus (A_1
\cup A_2) ,k+l).
\end{equation*}
And in all other cases, the square of the terms assures positivity. Let
$\sum_{A_1, A_2}$ be the sum over all sets $A_1 \subset H_1, A_2 \subset
H_3$ such that $|A_1| = 2$, $|A_2| = 1$ and,
\begin{equation*}
S_1(\bv, \delta) := \sum_{n=N} ^{2N} \sum_{A_1 , A_2}\Lambda_R(n, A_1 \cup
A_2 ,3)^2 [ \sum_{\substack{\HH \subset H \\ |\HH|=k }} \Lambda_R(n,\HH
\setminus (A_1 \cup A_2),k+l) ]^2.
\end{equation*}
This makes the bound we are considering (notice that we do not consider
the sets with multiplicity $k!$ in $S_1$),
\begin{equation*}
B_1(\bv, \delta) \leq S_1(\bv, \delta) (k!)^2 {36 \over \log^6 (R)}
\log(3N).
\end{equation*}
\subsection{Simplifying the Equation}

The goal of this section is to provide an estimate for $S_1(\bv, \delta)$.
Fix the sets $A_1$, $A_2$, then,
\begin{align*}
 \sum_{\substack{\HH \subset H \\ |\HH|=k }} \Lambda_R(n,\HH \setminus (A_1 \cup A_2),k+l)
= \sum_{j=0}^3 {3 \choose j}\sum_{\substack{\HH \subset H, \; |\HH|=k-j
\\ (A_1 \cup A_2) \cap \HH = \emptyset}} \Lambda_R(n,\HH ,k+l).
\end{align*}
Where the $\displaystyle{3 \choose j}$ terms result from the choice of
which elements of $A_1 \cup A_2$ were removed from the set under
consideration. Letting $\displaystyle \sum_{\HH_1, \HH_2} ^ {j_1, j_2}$ be
the sum over all $\HH_1, \HH_2 \subset H$ such that $|\HH_1|
= k - j_1$ and $|\HH_2| = k - j_2$ and $(\HH_1 \cup \HH_2) \cap (A_1 \cup
A_2) = \emptyset$ (where the dependence on $A_1$ and $A_2$ of the
summation is understood but not noted),
\begin{align*}
S_1(\bv, \delta) =   \sum_{n=N} ^{2N} \sum_{A_1 , A_2} \Lambda_R(n, A_1
\cup A_2 ,3)^2 \;  [ \sum_{j=0}^3 {3 \choose j}\sum_{\substack{\HH
\subset H, \; |\HH|=k - j \\ (A_1 \cup A_2) \cap \HH = \emptyset}}\hspace{-.3in}
\Lambda_R(n,\HH,k+l) ] ^2 =
\end{align*}

$ \displaystyle
\sum_{n=N} ^{2N} \sum_{A_1 , A_2}\Lambda_R(n, A_1 \cup A_2 ,3)^2 \times$

$\displaystyle \hfill\sum_{j_1, j_2=0}^3 {3 \choose j_1} {3 \choose j_2} \sum_{\HH_1,
\HH_2}^{j_1, j_2} \Lambda_R(n,\HH_1 ,k+l)\Lambda_R(n,\HH_2,k+l) = \vspace{.2in}
$

$\hfill \displaystyle \sum_{j_1, j_2 = 0}^3 \sum_{r = 0}^{\min(k-j_1, k-j_2)} {3 \choose j_1} {3 \choose j_2}
S_1'(\bv, \delta ,j_1,j_2, r).$

\noindent Where, \\

$\displaystyle S_1'(\bv, \delta ,j_1,j_2, r) := \vspace{.2in}$

$\displaystyle \hfill \sum_{n=N} ^{2N} \sum_{A_1 , A_2}
\sum_{\HH_1, \HH_2}^{j_1, j_2,r} \Lambda_R(n, A_1 \cup A_2 ,3)^2
\Lambda_R(n,\HH_1 ,k+l)\Lambda_R(n,\HH_2 ,k+l),$

\noindent and $\displaystyle \sum_{\HH_1, \HH_2} ^ {j_1, j_2, r}$ is the
sum over all $\HH_1, \HH_2 \subset \HH$ such that $|\HH_1| = k
- j_1$ and $|\HH_2| = k - j_2$, $(\HH_1 \cup \HH_2) \cap (A_1 \cup A_2) =
\emptyset$ and $|\HH_1 \cap \HH_2 | = r$ (once again there is an unnoted
dependence on $A_1$ and $A_2$). By {\it Theorem 1} (recall that if
$|\HH_1|=k-j_1$ and the third argument of the $\Lambda$ function is $k+l$,
the $l$ value increases to $l + j_1$),
\\

\noindent $\displaystyle S_1'(\bv, \delta ,j_1,j_2, r) =N {2l +j_1 + j_2 \choose l + j_1} {(\log R) ^ 3 (\log R)^{2l + r + j_1 +
j_2} \over 3! (2l + r + j_1 + j_2)!} \times $

$\displaystyle \hfill \sum_{A_1, A_2} \sum_{\HH_1,
\HH_2}^{j_1, j_2, r} (\gs(A_1 \cup A_2 \cup \HH_1 \cup \HH_2)  +  o_M (1)).$

Since $|H_1| = |H_3| = \delta |H|$ where $ \delta $ is an absolute constant , the growth condition for {\it
Theorem 2} is satisfied where the subintervals are taken as $[B_1, C_1] = H_1$, $[B_2, C_2] = H_2$, $[B_3,
C_3]$,$[B_4, C_4]$, $[B_5, C_5] = H$ and $2$ elements are taken from $[B_1, C_1]$, one from $[B_2, C_2]$, $k -
j_1 - r$ taken from $[B_3, C_3]$, $k - j_2 - r$ from $[B_4, C_4]$ and $r$ from $[B_5, C_5]$. The elements from
the first interval correspond to the elements of $A_1$, the second of $A_2$, the third of $\HH_1 \setminus \HH_2$,
the fourth of $\HH_2 \setminus \HH_1$ and the elements from the fifth interval to the elements of $\HH_1 \cap \HH_2$.
We can almost use the theorem, except that we have the additional restriction that each of the five sets are
disjoint (By definition the $A_i$ sets are disjoint from each other and any $\HH_i$. Obviously the sets $\HH_1
\setminus \HH_2$, $\HH_1 \cap \HH_2$ and $\HH_2 \setminus \HH_1$ are disjoint). Since we showed in the proof of {\it
Theorem 2} that each of the singular series is bounded above by a constant power of $\log h$, the sum of the
singular series where the sets are not disjoint is $\ll_M (\log^{C_M}h) h^{2k-r-j_1-j_2+2}$. It follows from the
Theorem and the previous fact that the sum over disjoint sets we consider dominate the magnitude of the entire
sum that:

$ \displaystyle
\sum_{A_1, A_2} \sum_{\HH_1, \HH_2}^r \gs(A_1 \cup A_2 \cup \HH_1 \cup
\HH_2) = {\delta^3 h^{2k-r -j_1 - j_2 + 3} \over 2! (k - j_1 -r)! (k- j_2 -r)! r!} \times $

$\hfill \left(1 + o_M(1) \right).
$

\noindent Therefore, \vspace{.2in}

$S_1'(\bv, \delta ,j_1,j_2, r) = \displaystyle
N {2l +j_1 + j_2 \choose l + j_1} {(\log R) ^ 3 \over 6} {(\log R)^{2l + r
+ j_1 + j_2} \over (2l + r + j_1 + j_2)!} \times \vspace{.1in} $

$\hfill \displaystyle {\delta ^3 h^{2k-r - j_1
- j_2  + 3} \over 2 (k-j_1-r)! (k-j_2 -r)! r!}(1 + o_M(1)).
$

It is now possible to step back to evaluate $S_1(\bv, \delta)$. Define the
following notation (with an empty product being defined as $1$):

$$
\gamma (j_i,k,r) = (k-r)(k-r-1) \ldots (k-r-j_i+1),$$ $$ \beta(j_1,
j_2,l,r) = (r+2l+1) (r+2l+2)\ldots (r+2l+j_1+j_2),
$$
$$
 a(j_1, j_2,l) = {2l +j_1 +
j_2 \choose l + j_1} {2l \choose l}  ^ {-1},$$ $$ \mu (j_1, j_2,k,l,r) =
{\gamma(j_1,k,r) \gamma(j_2,k,r) a(j_1,j_2,l) \over \beta(j_1, j_2,r,l)}
{3 \choose j_1} {3 \choose j_2}.
$$

We define this in order to put our previous derivation into a form more comparable with the work of
\cite{goldston-2005}. Notice,

$\displaystyle \left({k! k! \over (k-j_1-r)! (k-j_2-r)! r! (2l +r +j_1 +j_2)!}
\right) = $

\begin{align*} &   {k \choose r}^2 {r! \over (2l + r)!} \left({\gamma(j_1,k,r)
\gamma(j_2,k,r) \over \beta(j_1, j_2,r,l)} \right) =\\
&  \displaystyle {k \choose r}^2
{1 \over (r+1)(r+2) \ldots (r +2l)} \left({\gamma(j_1,k,r) \gamma(j_2,k,r)
\over \beta(j_1, j_2,r,l)} \right). \end{align*}

With this, it is possible to restate the $S_1$ function. Defining $x =
\log R / h$, and $\chi := N (\log R)^{2l +6} h^{2k} / (k!)^2$,

\begin{eqnarray*}
\lefteqn{S_1(\bv, \delta) \sim} \\
& & \chi {2l  \choose l }\sum_{j_1,
j_2 =0}^3 \sum_{r=0}^{\min(k-j_1, k-j_2)} {k \choose r}^2 {\delta^3
x^{r+j_1+j_2-3}\mu(j_1, j_2,k,l,r) \over 12 (r+1 ) \ldots (r+2l)}  \leq \\
& & \chi{2l  \choose l }\sum_{j_1, j_2
=0}^3 \sum_{r=0}^{k} {k \choose r}^2 {\delta^3 x^{r+j_1+j_2-3}\mu(j_1,
j_2,k,l,r) \over 12
(r+1 ) \ldots (r+2l)}  = \\
& & \chi{2l  \choose l }
\sum_{r=0}^{k} {k \choose r}^2 { x^{r} \over (r+1 )\ldots (r+2l)}
\sum_{j_1, j_2 =0}^3 {\delta^3 x^{j_1+j_2-3} \mu(j_1, j_2,k,l,r) \over 12
}.
\end{eqnarray*}

And therefore, an upper bound for our original sum is :\\

$ \displaystyle B_1(\bv, \delta) = {\sum_{n=N+1}^{2N}}^* \left(
\sum_{|\HH|= k} \Lambda_R(n, \HH, k+l) \right) ^2 \log (3N) \leq$ \vspace{.1in}

$\displaystyle N (\log
R)^{2l} h^{2k} \log(3N) {2l  \choose l } \sum_{r=0}^{k}
{k \choose r}^2 { x^{r} \over (r+1 )\ldots (r+2l)} \times$

$\hfill \displaystyle  \sum_{j_1, j_2 =0}^3 {3
\delta^3 x^{j_1+j_2-3} \mu(j_1, j_2,k,l,r)  } (1 + o_M(1)) \sim$

$\hspace{-.12in}\displaystyle N  h^{2k+1} {2l  \choose l } (\log R)^{2l}
\sum_{r=0}^{k} {k \choose r}^2 { x^{r} \over (r+1 )\ldots (r+2l)}\times $

$\hfill \displaystyle \left[
{3 \delta^3 \over \Theta x^2} \right]\sum_{j_1, j_2 =0}^3 {
x^{j_1+j_2} \mu(j_1, j_2,k,l,r) }$,\\

\noindent with identical reasoning giving the same bound for $B_2(\bv,
\delta)$. In \cite{goldston-2005} they derive the two facts that:

$\displaystyle \sum_{n=N+1}^{2N} 2\log N \left(\sum_{|\HH|=k}
\Lambda_R(n,\HH,k+l) \right)^2 \sim $ \vspace{.1in}

$\displaystyle \hfill 2 N h^{2k} \log (3N) {2l \choose l} (\log R)^{2l}
\sum_{r=0}^{k} {k \choose r}^2 {x^r \over (r+1) \ldots (r+2l)}$, \vspace{.2in}

$\displaystyle \sum_{n=N+1}^{2N} \left(\sum_{1 \leq h_0 \leq h}
\theta(n+h_0) \sum_{|\HH|=k} \Lambda_R(n,\HH,k+l) \right)^2 \sim $ \vspace{.1in}

$\displaystyle \hfill N h^{2k+1} {2l \choose l} (\log R)^{2l}
\sum_{r=0}^{k} {k \choose r}^2 {x^r \over (r+1) \ldots (r+2l)} \left(
{2a(1,0,l)k \over r + 2l + 1}x + 1 \right)$.

It follows that if after factoring out: $${2l \choose l}N h^{2k+1} (\log
R)^{2l}$$ (which controls all dependence on $N$) the remaining factor is
positive, then $F_2 \leq \lambda (1 - \delta)$. Therefore, it suffices to show
(where the $3$ in the final term becomes a $6$ because we are considering
both $B_1$ and $B_2$):

$ \displaystyle
\sum_{r=0}^{k} {k \choose r}^2 {x^r \over (r+1) \ldots (r+2l)}\times$

$\displaystyle \hfill \left(
{2a(1,0,l)k \over r + 2l + 1}x + 1 - {2x \over \Theta}  - {6 \delta^3
\over \Theta x^2} \sum_{j_1, j_2 =0}^3 x^{j_1+j_2} \mu(j_1, j_2,k,l,r)
\right)
$

\noindent is positive. This will imply $F_2 \leq  \lambda ( 1 - \delta)$.

\subsection{The Derivation}
First notice that by bounding each component individually,
$$\mu (j_1, j_2,k,l,r) \leq 2^{j_1 + j_2} \left( k-r \over r \right) ^{j_1 + j_2} {3 \choose j_1} {3 \choose j_2}$$
\noindent which implies,
$$\sum_{j_1, j_2 =0}^3
x^{j_1+j_2} \mu(j_1, j_2,k,l,r) \leq \left( \sum_{j=0}^3 {3 \choose j} 2^j
x^j \left( k-r \over r \right)^j  \right)^2 = \left( 2x {k-r \over r} +1
\right)^6.$$

There are now two parts of the proof remaining. Let (similarly to
\cite{goldston-2005}),

$$
f(r) = {k \choose r}^2 {x^r \over (r+1) (r+2) \ldots (r+2l)},
$$

$$
P(r, \delta) =  {2a(1,0,l)k \over r + 2l + 1}x + 1 - {2x \over \Theta}  -
{6 \delta^3 \over \Theta x^2} \left( 2x {k-r \over r} + 1 \right)^6. $$

We will show that with good choices for $k,l,\lambda, \Theta, \delta$,
$\sum_{r = 0} ^k f(r)P(r, \delta)$ will be positive, giving the bound that
$F_2 \leq \lambda(1 - \delta) $. In a sense, $f(r)$ will contribute to the
magnitude of the $r^{th}$ term while $P(r, \delta)$ will control the sign.
First, notice that the term $f(r)$ is maximized when (one can justify this
heuristic by a little computation, but it is not necessary, the decay when
the terms are much bigger or smaller will be shown shortly),
$$r \sim {k \over z+1} \text{ where } z = {1 \over \sqrt{x}}\text{. Let } r_0 = \left[ k \over z + 1 \right].$$

We would now like to find with a given choice of $x$, what the maximal choice of $\delta$ can be such that the
maximal term is positive (it is possible that a given $x$ will have no valid corresponding $\delta$ if the
interval we are considering is simply too small. In these cases there will be a contradiction with the
definition of $\delta$, giving a value that is not in $[0, 1/2)$ as $\delta$ was defined to be). If we took
$\Theta = 1/4$ when in reality it can only be taken to be $1/4 - \epsilon$, $k/l = k$, when in reality it can
only be taken as $k ( 1 - \epsilon)$ (in applications in \cite{goldston-2005} $l = o(k)$)one has, in a sense,
the function $P(r, \delta)$ `approaches'. Call this function $P' (r, \delta)$ and notice:
$$P' \left({ k \over z+1}, \delta \right) = {4k \over { k \over z+1}}x + 1 - {8x} - {24 \delta^3 \over x^2} \left( 2x {k-{ k \over z+1} \over { k \over z+1}} + 1 \right)^6 > 0 \Leftrightarrow $$
\begin{center}
(Since $x \neq 0$ in this application, we can divide by it. Also, recall $x = z^{-2}$) \end{center}
$$  4 (z + 1) + z^2 - 8  - 24 \delta^3 \left( z + 2  \right)^6 > 0 \Leftrightarrow$$
$$ ( z + 2)^2 - 8 - 24 \delta^3 \left( z + 2  \right)^6 > 0.$$
Which is equivalent with:
\[ \delta ^3 \leq {( z + 2)^2 - 8 \over 24 \left( z + 2  \right)^6}.
\]

So, if $\delta$ is exactly the cube root of the value on the right above, $P'(k / (z + 1))=0$. If it is
below, one can check easily that $P'(k/(z+1))$ is positive.

\begin{theorem}
Let $\lambda > 0$, $z = 2 \sqrt{ \lambda }$ and $\delta' = \displaystyle \sqrt[3]{( z + 2)^2 - 8 \over 24 \left(
z + 2  \right)^6}$. Then, if $\delta' \in [0, 1/2)$, we have $F_2 \leq \lambda (1 - \delta')$.
\end{theorem}
\begin{proof}

Fix $\lambda$ and $\delta'$ for the rest of the proof that satisfy the above conditions (any constants from here on may depend on $\lambda$ and $\delta$). We claim that for any $\delta < \delta' $ the total sum is positive assuming the $k,l$ values are sufficiently large with $l = o(k)$ and $\Theta = (1/4)(1 - (1/l))$. This will imply that $F_2 \leq \lambda (1 - \delta)$. Since $\delta$ can be taken arbitrarily close to $\delta'$ this will imply $F_2 \leq \lambda ( 1 - \delta')$. \\

\begin{lemma}
Let $\delta = \delta' \sqrt[3]{1- \epsilon}$. Then, there is a small constant $\nu > 0$, which depends on
$\epsilon$ such that for $k,l > C (\epsilon)$ and $l/k < c (\epsilon)$, we have $P(r) > c'(\epsilon)>0$ for $r \in [r_0
- \nu k, r_0 + \nu k]$.
\end{lemma}

\begin{proof}
Take $r \in [r_0 - \nu k, r_0 + \nu k]$, we will analyze each term of
$P(r, \delta)$ separately (recalling that we have set $x$ as a constant
 $+ O(l^{-1})$ ($x = \Theta / \lambda$) and $k / r_0$ as a constant $+
O(k^{-1})$). As $k, l \to \infty$, $l = o(k)$ and $\nu \to 0$,

$$ {2a(1,0,l)k \over r + 2l + 1}x  \geq {4 \over r_0/k} x + O(l/k) + O(l^{-1}) + O(\nu),
$$
$$1 - {2x \over \Theta} = 1 - 8x + O(l^{-1}),
$$
$ \displaystyle
{6 \delta^3 \over  x^2 \Theta} \left( 2x {k-r \over r} + 1 \right)^6 \leq 
{24 \delta'^3 \over  x^2}\left( 2x \left({k \over r_0 } -1\right) + 1
\right)^6 + O(\nu) - c (\epsilon)+ O(l^{-1}) ,
$

\noindent with $c(\epsilon) > 0$. Combining these, one sees a
correlation with $P'(k/z+1, \delta')$,
\begin{align*}
P(r, \delta) &\geq P'(k / z+1, \delta') + c (\epsilon) + O(l/k) + O(l^{-1}) + O(k^{-1}) + O(\nu) \\
&\geq c (\epsilon) + O(l/k) + O(l^{-1}) + O(k^{-1}) + O(\nu)
\end{align*}

\noindent which proves the statement. We first fix $\epsilon$ and then
select $l,k$ big enough and $\nu$ small enough such that $P(r, \delta)
\geq c (\epsilon)/2$ for all $r \in [r_0 - \nu k, r_0 + \nu k]$.
\end{proof}

The rest of the proof will proceed in the following manner: First we will
show that for $r < r_0 - {\nu \over 2} k = r_1$ or $r > r_0 + {\nu \over
2}k = r_2$, the values $f(r)$ rapidly decrease by at least a constant
factor in magnitude. This will imply the negative terms (which are smaller
than $r_0 - \nu k$ and greater than $r_0 + \nu k$) will all be
exponentially small in $k$ and their total sum can be bounded by the $r_0$
term. For notational simplicity, let $\nu' = \nu/2$.

First, we begin analyzing the terms below $r_1$. Take any $ r < r_1$.
\begin{eqnarray*}
{f(r+1) \over f(r)} &=& \left( k-r \over r+1 \right) ^2 { x (r+1) \over r+2l
+1} > \left( k-r \over r+1 \right) ^2 { x r \over r+2l
+1}.
\end{eqnarray*}

Notice first that, since $r < ck$ where $c < 1$:

$\displaystyle {(k-r)(r) \over (r+ 2l + 1)(r+1)}=$
\begin{align*}
  & \left( { {r^2 \over kr -
r^2} + {(2l+2)r \over kr - r^2} + {2l + 1 \over kr - r^2}
}\right)^{-1} = \left( { {r \over (k - r)} + O(l/k) }\right)^{-1} \geq \\
&   \left( { r_0 - \nu' k \over k - r_0 + \nu' k} + O(l / k)
\right)^{-1}
 =  \left( { k - r_0 + \nu' k  \over r_0 - \nu' k}\right) ( 1 + O(l/k)) \geq\\
&  \left( {k \over r_0}(1 + c(\nu)) - 1 \right) (1 + O(l/k)) \text{ where } c(\nu) > 0 \\
& \geq  \left( {k \over r_0} - 1\right) ( 1 + c(\nu)/2) \text{ for }k/l
\text{ small enough. Secondly,}
\end{align*}

$\displaystyle{k-r \over r + 1} \geq {k - r_0 \over r_0} \text{ for } k ,
l \text{ sufficiently large. Therefore,}$
\begin{eqnarray*}
{f(r+1) \over f(r)} &>& \left( k-r \over r+1 \right) ^2 { x r \over r+2l +1} \;\;\; \text{ which by the definition of } r_0 \text{ is,}\\
& \geq & \left(k - r_0 \over r_0 \right) ^2 x (1 + c(\nu)/2)  = 1 +
c'(\nu)/2 + O(k^{-1}) + O(l^{-1}).
\end{eqnarray*}
So, as the terms go below the $r_1 ^{\text{st}}$ the ratio between any two terms decreases by a constant factor.
Therefore, the sum of all terms below $r_0 - \nu k$ is a polynomial in $k$ (there are at most $k$ such terms and
$P(r, \delta)$ is bound easily by a polynomial in $k$) times an inverse exponential in $k$ times the $r_0$ term.
The $r_0$ term will therefore be greater in magnitude than (any constant multiple times) the sum of all negative
terms $r$ with $r < r_1$ for $k$ sufficiently large. Now we will show the same holds with $r
> r_2$ (let $c(\nu)$ and $c'(\nu)$ denote small positive constants
depending on $\nu$), \\

$ \displaystyle
{f(r+1) \over f(r)} \leq \left( k-r \over r+1 \right) ^2 x
\leq \left({k \over r} - 1 \right)^2 x \leq \left( {k \over r_0 + k\nu' } - 1 \right)^2 x \leq \\$

$ \displaystyle
\left( {k \over r_0} (1 - c(\nu)) - 1 \right)^2 x  \leq (1 - c'(\nu)) \left( {k \over r_0} -1\right)^2 x = 1 - c'(\nu) +
O(k^{-1} + l^{-1}).
$

Therefore, the magnitude of these terms decay exponentially. Since, as
before $P(r, \delta)$ is bounded above by a polynomial in $k$ and $f(r)$
is bounded above by an inverse exponential in $k$ when compared to the
$r_0$ term, the sum of all such $r > r_2$ can easily be bounded under half
of the magnitude of the term at $r_0$. This completes the proof of
positivity, which implies the theorem.
\end{proof}

Notice that if we take $\lambda = (\sqrt{2} - 1)^2$, this implies $\delta
= 0$ and the necessary condition is satisfied. This is precisely the
result implied by \cite{goldston-2005}. One can check that if we increase
$\lambda$ by a very small amount so that the $\delta'$ value increases and
stays within the allotted interval (since it varies continuously with
respect to $\lambda$) the value $\lambda ( 1 - \delta')$ will initially
decrease. This proves the theorem. One can numerically check that taking
$\lambda = .172$ implies $\delta' \sim .007794$ which implies $F_2 < .172
(1 - .007794) < .17066$.

\bibliographystyle{plain}
\bibliography{Works}

\end{document}